\documentclass[12pt,a4paper]{amsart}

\usepackage{amscd}
\usepackage{amssymb}
\usepackage{mathtools}
\usepackage{hyperref}
\usepackage{csquotes}

\calclayout

\hypersetup{
	citecolor=black,
    colorlinks=true,
    linkcolor=black,
    filecolor=magenta,      
    urlcolor=black,
}

\theoremstyle{definition}
\newtheorem{definition}{Definition}[section]

\newtheorem{lemma}[definition]{Lemma}

\begin{document}

\title{A Note on \enquote{Spaces of Topological Complexity One}}

\author{Ramandeep Singh Arora}

\address{Department of Mathematical Sciences\\ Indian Institute of Science Education and Research (IISER) Mohali\\ Sector 81\\ S.A.S. Nagar\\ P.O. Manauli\\ Punjab 140306\\ India.}

\email{ms14030@iisermohali.ac.in}


\begin{abstract}
   Here we give a reformulation of a key lemma in the paper \cite{grant}, \enquote{Spaces of Topological Complexity One}, which is necessary due to an oversight. 
\end{abstract}

\maketitle

\section{Introduction}
The topological complexity \text{TC}$(X)$ of motion planning, introduced by M. Farber in \cite{farber}, is a numerical homotopy invariant which measures the discontinuity of motion planning in the space $X$. 

This note emerged out while reviewing the research paper \cite{grant} as a part of my MS-thesis on topological complexity. Lemma $3.3$ in \cite{grant} consists of two parts in which the second is proved using the first. It can be seen that the statement of the first part is not valid for $n \geq 4$. Thus the proof given in \cite{grant} has a flaw. In addtion, the second part of the lemma plays a key role in proving the main result of \cite{grant}. However, it is important to note that statement of the second part is correct and thus the main result of the paper is still worthy of praise. In this note we are going to describe the problem in the proof of the lemma, give a direct proof of the second part and show that with a simple modification the first part can be corrected.

\section{Oversight in the Lemma}

Let $X$ be a topological space and $\Bbbk$ be a field. Let $H^*(X;\Bbbk)^{\otimes n}$ denote the $n$-fold tensor product $H^*(X;\Bbbk)$ over $\Bbbk$. Suppose $a \in H^*(X;\Bbbk)$. Let $a_i$ denote the element $1 \otimes \cdots \otimes 1 \otimes a \otimes 1 \otimes \cdots \otimes 1$ of $H^*(X;\Bbbk)^{\otimes n}$ where $a$ is at the $i^{th}$ position and $|a|$ denote the degree of the cohomology class $a$ in $H^*(X;\Bbbk)$.

\vspace{1mm}
Suppose $a,b \in H^*(X;\Bbbk)$. Due to the graded product on $H^*(X;\Bbbk)^{\otimes n}$, we have
\begin{equation*}
a_ib_j=
\begin{cases}
1 \otimes \cdots \otimes 1 \otimes a \otimes 1 \otimes \cdots \otimes 1 \otimes b \otimes 1 \otimes \cdots \otimes 1 & \text{if } i \leq j\\
(-1)^{|a||b|}(1 \otimes \cdots \otimes 1 \otimes b \otimes 1 \otimes \cdots \otimes 1 \otimes a \otimes 1 \otimes \cdots \otimes 1) & \text{if } i > j.
\end{cases}
\end{equation*}
where $a$ and $b$ are at the $i^{th}$ and $j^{th}$ position respectively. Moreover, we have 
\begin{equation*}
a_ib_j=(-1)^{|a||b|}b_ja_i.
\end{equation*}

\pagebreak

\begin{lemma} (Lemma $3.3$ in \cite{grant})
\label{incorrectlemma}
Suppose we have $a,b \in H^*(X;\Bbbk)$. With the above notation, for $n \geq 2$ we have 
\begin{equation}
\label{equation1}
(a_1-a_2)(a_1-a_3) \cdots (a_1 -a_n) \equiv (-1)^n(a \otimes 1-1 \otimes a) \otimes a \otimes \cdots \otimes a
\end{equation}
modulo terms in the ideal of $H^*(X;\Bbbk)^{\otimes n}$ generated by the elements $a^2 \otimes 1 \otimes \cdots \otimes 1$ and $a \otimes \cdots \otimes a \otimes 1$. Consequently, we have
\begin{equation*}
(b_1-b_2)(a_1-a_2)(a_1-a_3) \cdots (a_1-a_n) \equiv (-1)^{n+1}(b\otimes a + (-1)^{|a||b|}a \otimes b)\otimes a \otimes \cdots \otimes a
\end{equation*}
modulo terms in the ideal of $H^*(X;\Bbbk)^{\otimes n}$ generated by the elements $a^2 \otimes 1 \otimes \cdots \otimes 1$, $ba \otimes 1 \otimes \cdots \otimes 1$, and $1 \otimes ba \otimes 1 \otimes \cdots \otimes 1$.
\end{lemma}

\vspace{2mm}
Counter example:
\vspace{1mm}

Let us observe the equation (\ref{equation1}) for the case $n=4$. 
\begin{equation*}
(a_1-a_2)(a_1-a_3)(a_1-a_4) \equiv (a_1 -a_2)(-a_1a_4-a_3a_1+a_3a_4)
\end{equation*}
as $a_1^2= a^2 \otimes 1 \otimes 1 \otimes 1$
\begin{equation*}
(a_1-a_2)(a_1-a_3)(a_1-a_4) \equiv a_1a_3a_4 + a_2a_1a_4 + a_2a_3a_1 -a_2a_3a_4
\end{equation*}
as $a_1 a_1 a_4=a^2\otimes 1 \otimes 1 \otimes a$ and $a_1a_3a_1=(-1)^{|a|^2}a^2 \otimes 1 \otimes a \otimes 1$ are multiples of $a^2 \otimes 1 \otimes 1 \otimes 1$. Moreover, $a_2a_3a_1 = a \otimes a \otimes a \otimes 1$ which is a generator of the ideal. Thus
\begin{equation*}
\begin{split}
(a_1-a_2)(a_1-a_3)(a_1-a_4) \equiv & a_1a_3a_4 + a_2a_1a_4 -a_2a_3a_4\\
\equiv &  a \otimes 1 \otimes a \otimes a + (-1)^{|a|^2} a \otimes a \otimes 1 \otimes a - 1 \otimes a \otimes a \otimes a\\
\equiv & (-1)^4(a \otimes 1 - 1 \otimes a) \otimes a \otimes a + (-1)^{|a|^2} a \otimes a \otimes 1 \otimes a
\end{split} 
\end{equation*}
In the above equation we have an extra term $(-1)^{|a|^2}a \otimes a \otimes 1 \otimes a$ which clearly doesn't lie in the ideal generated by $a^2 \otimes 1 \otimes 1 \otimes 1$ and $a \otimes a \otimes a \otimes 1$. Thus the statement of the first part of Lemma $3.3$ in \cite{grant} is not correct.

\vspace{5mm}
Oversight in the proof:
\vspace{2mm}

Let $I_k$ be the ideal of $H^*(X;\Bbbk)^{\otimes n}$ generated by $a^2 \otimes 1 \otimes \cdots \otimes 1$ and $a \otimes \cdots \otimes a \otimes 1 \otimes \cdots \otimes 1$, wherein the latter term we have $(k-1)$ occurrences of $a$. Firstly, we observe that the ideal $I_{k+1} \subset I_{k}$.

The incorrect step in the proof is the inducting step. By induction hypothesis we have 
\begin{equation*}
(a_1-a_2)\cdots(a_1-a_k) \equiv (-1)^k(a \otimes 1 -1 \otimes a) \otimes a \otimes \cdots \otimes  a \otimes 1 \otimes \cdots \otimes 1 \text{ mod } I_k,
\end{equation*}
where both the terms have $(k-1)$ occurrences of $a$. Using the above equation we have $((a_1-a_2)\cdots(a_1 -a_k))(a_1-a_{k+1})$
\begin{equation*}
\begin{split}
\equiv & (-1)^k((a \otimes 1 -1 \otimes a) \otimes a \otimes \cdots \otimes  a \otimes 1 \otimes \cdots \otimes 1) (a \otimes 1 \otimes \cdots \otimes 1 ) \\
- & (-1)^k((a \otimes 1 -1 \otimes a) \otimes a \otimes \cdots \otimes  a \otimes 1 \otimes \cdots \otimes 1)(1 \otimes \cdots \otimes 1 \otimes a \otimes 1 \otimes \cdots \otimes 1)
\end{split}
\end{equation*}
mod $I_k$. When we modulo this expression with the ideal $I_{k+1}$ the first part of the expression $((a \otimes 1 -1 \otimes a) \otimes a \otimes \cdots \otimes  a \otimes 1 \otimes \cdots \otimes 1) (a \otimes 1 \otimes \cdots \otimes 1)$ vanishes but we may get some extra terms since $I_{k+1} \subset I_k$. Thus
\begin{equation*}
(a_1-a_2)\cdots(a_1-a_{k+1}) \not\equiv (-1)^{k+1}(a \otimes 1 -1 \otimes a) \otimes a \otimes \cdots \otimes  a \otimes 1 \otimes \cdots \otimes 1 \text{ mod } I_{k+1}.
\end{equation*} 

\section{Reformulation of the Lemma}

Now we are going to give an alternate proof of the second part of the lemma and then move on to give a modification of the first part of the lemma. 

\begin{lemma}
Suppose we have $a,b \in H^*(X;\Bbbk)$. For $n \geq 2$ we have 
\begin{equation*}
(b_1-b_2)(a_1-a_2)(a_1-a_3) \cdots (a_1-a_n) \equiv (-1)^{n+1}(b\otimes a + (-1)^{|a||b|}a \otimes b)\otimes a \otimes \cdots \otimes a
\end{equation*}
modulo terms in the ideal of $H^*(X;\Bbbk)^{\otimes n}$ generated by the elements $a^2 \otimes 1 \otimes \cdots \otimes 1$, $ba \otimes 1 \otimes \cdots \otimes 1$, and $1 \otimes ba \otimes 1 \otimes \cdots \otimes 1$.
\end{lemma}

\begin{proof}
A term in the expression $(b_1-b_2)(a_1-a_2)(a_1-a_3) \cdots (a_1-a_n)$ will be of the form 
\begin{equation}
\label{expression}
(-1)^{\epsilon} \text{ } b_j a_{i_1} \cdots a_{i_{n-1}}
\end{equation}
where $j \in \{1,2\}$, $i_k \in \{1, k+1\}$ for $k\in \{1, 
\ldots ,n-1\}$ and $\epsilon \in \mathbb{Z}$.

If $a_1$ occurs twice in the expression (\ref{expression}), then the term itself is a multiple of $a^2 \otimes 1 \otimes \cdots \otimes 1$. Thus we can assume $a_1$ occurs atmost once.

Let us assume $b_i=b_1$. If $a_1$ comes in the expression (\ref{expression}), then the term itself is a multiple of $ba \otimes 1 \otimes \cdots \otimes 1$. Thus we can assume $a_1$ doesn't occur in the expression and hence the only expression possible is $(-1)^{n-1}b_1a_2a_3\cdots a_n$. Similarly, if we assume $b_i =b_2$, then the only expression possible is $(-1)^{n-1}b_2a_1a_3\cdots a_n$ since $a_1$ can occur atmost once.

Therefore we get 
\begin{equation*}
(b_1-b_2)(a_1-a_2)(a_1-a_3) \cdots (a_1-a_n) \equiv (-1)^{n-1}(b_1a_2a_3\cdots a_n+b_2a_1a_3\cdots a_n)
\end{equation*}
modulo terms in the ideal of $H^*(X;\Bbbk)^{\otimes n}$ generated by the elements $a^2 \otimes 1 \otimes \cdots \otimes 1$, $ba \otimes 1 \otimes \cdots \otimes 1$, and $1 \otimes ba \otimes 1 \otimes \cdots \otimes 1$.
\end{proof}

\begin{lemma}
Suppose we have $a \in H^*(X;\Bbbk)$. For $n \geq 2$ we have 
\begin{equation*}
(a_1-a_2)(a_1-a_3) \cdots (a_1 -a_n) \equiv (-1)^n(a \otimes 1-1 \otimes a) \otimes a \otimes \cdots \otimes a
\end{equation*}
modulo terms in the ideal of $H^*(X;\Bbbk)^{\otimes n}$ generated by the elements $a^2 \otimes 1 \otimes \cdots \otimes 1$ and $a \otimes a \otimes 1 \otimes \cdots \otimes 1$.
\end{lemma}

\begin{proof}
A term in the expression $(a_1-a_2)(a_1-a_3) \cdots (a_1-a_n)$ will be of the form 
\begin{equation}
\label{expression2}
(-1)^{\epsilon} \text{ } a_{i_1} \cdots a_{i_{n-1}}
\end{equation}
where $i_k \in \{1, k+1\}$ for $k\in \{1, 
\ldots ,n-1\}$ and $\epsilon \in \mathbb{Z}$.

If $a_1$ occurs twice in the expression (\ref{expression2}), then the term itself is a multiple of $a^2 \otimes 1 \otimes \cdots \otimes 1$. Thus we can assume $a_1$ occurs atmost once.

Let us assume $a_{i_1}=a_2$. If $a_1$ comes in the expression (\ref{expression2}), then the term itself is a multiple of $a \otimes a \otimes 1 \otimes \cdots \otimes 1$. Thus we can assume $a_1$ doesn't occur in the expression and hence the only expression possible is $(-1)^{n-1}a_2a_3\cdots a_n$. Similarly, if we assume $a_{i_1}=a_1$, then the only possible expression is $(-1)^{n-2}a_1a_3 \cdots a_n$ since $a_1$ can occur atmost once. 

Therefore we get 
\begin{equation*}
(a_1-a_2)(a_1-a_3) \cdots (a_1-a_n) \equiv (-1)^{n}(a_1a_3\cdots a_n-a_2a_3\cdots a_n)
\end{equation*}
modulo terms in the ideal of $H^*(X;\Bbbk)^{\otimes n}$ generated by the elements $a^2 \otimes 1 \otimes \cdots \otimes 1$ and $a \otimes a \otimes 1 \otimes \cdots \otimes 1$.
\end{proof}

\paragraph{\textbf{Acknowledgement.} I would like to thank my advisor Dr. Mahender Singh of the Mathematics department at the Indian Institute of Science Education and Research (IISER) Mohali for his valuable guidance and introducing me to the research area of topological complexity. I would like to thank M. Grant, G. Lupton, and J. Oprea for their comments and positive feedback concerning this note. I would also like to thank IISER Mohali and KVPY for providing me fellowship during my BS-MS program.}

\end{document}